\documentclass[leqno]{amsart}%
\usepackage{amsmath}
\usepackage{amsfonts}
\usepackage{amssymb}
\usepackage{graphicx}%
\providecommand{\U}[1]{\protect\rule{.1in}{.1in}}
\newtheorem{theorem}{Theorem}

\newtheorem{corollary}[theorem]{Corollary}

\newtheorem{lemma}[theorem]{Lemma}

\newtheorem{proposition}[theorem]{Proposition}
\newtheorem{remark}[theorem]{Remark}

\newcommand{\arccosh}{\operatorname{arccosh}}
\newcommand{\mb}[1]{{\mathbf{#1}}}
\newcommand{\mbx}{{\mathbf{x}}}
\newcommand{\mby}{{\mathbf{y}}}
\newcommand{\mbz}{{\mathbf{0}}}
\newcommand{\arcsinh}{\operatorname{arcsinh}}

\renewcommand{\H}{\mathbb{H}}

\newcommand{\sect}{\operatorname{Sect}}

\newcommand{\vol}{\operatorname{Vol}}

\newcommand{\B}{{\mathbb{B}}}

\newcommand{\R}{\mathbb{R}}

\newcommand{\dist}{\operatorname{dist}}

\newcommand{\e}{{\varepsilon}}

\newcommand{\dih}{\mathrm{d}_{\H^n}}
\newcommand{\diht}{\mathrm{d}_{\H^3}}
\begin{document}

\title[Convex sets in the hyperbolic space]{Boundary structure of convex sets in the hyperbolic space}
\author{Giona Veronelli}
\address{Universit\'e Paris 13, Sorbonne Paris Cit\'e, LAGA, CNRS ( UMR 7539)
99\\
avenue Jean-Baptiste Cl\'ement F-93430 Villetaneuse - FRANCE} \email{veronelli@math.univ-paris13.fr}
\date{\today}

\begin{abstract}
We prove some results concerning the boundary of a convex set  in $\H^n$. This includes the convergence of curvature measures under Hausdorff convergence of the sets, the study of normal points, and, for convex surfaces, a generalized Gauss equation and some natural characterizations of the regular part of the Gaussian curvature measure.
\end{abstract}

\maketitle

\tableofcontents

\section{Introduction}
In this short note, we discuss some properties of the boundary of convex sets with non-empty interior in the real $n$-dimensional hyperbolic space $\H^n$. Most of the results we will present are well-known for convex sets in the Euclidean space $\R^n$ (see \cite{Schn} for a general reference on the subject) and hence they are naturally expected to be true in $\H^n$ too. However, we have not found detailed proofs in the literature.

First, we prove that the curvature measures introduced by Kohlmann via a Steiner polynomial type formula, \cite{Ko-GeoDed}, are stable under Hausdorff convergence of convex sets (see Section \ref{sect_conv}). In dimension 2, this can be applied to get the validity of a generalized Gauss-Codazzi equation for the boundary of a convex set in $\H^3$, giving as in the smooth case the equivalence of the 2nd curvature measure and of the intrinsic curvature measure (in the sense of BIC surfaces) up to a constant factor (see Theorem \ref{th_Gauss}). This latter in turn guarantees that Kohlmann's 2nd curvature measure we use here is the same as the extrinsic measure of convex sets of $\H^3$ introduced by Alexandrov, whose construction exploits the local euclidean character of $\H^n$, \cite[p 397]{Al}. 

In a second part we will deal with the regularity of boundary points, proving that a.e. point is normal. This roughly means that the neighborhood of a.e. point in the boundary can be suitably approximated by a smooth surface. This permits to define a concept of local curvature $\mathrm{LocCurv}(q)$ at every normal point $q$ (see Corollary \ref{coro_biLip}). Using Gauss equation and the smooth approximation around normal points we will finally proof that $\mathrm{LocCurv}(q)$ is the regular part, given by Lebesgue decomposition theorem, of the curvature measure (see Section \ref{sect_reg}).
Since every surface with lower bounded curvature is locally isometric to the boundary of a convex set in a space-form, the local curvature $\mathrm{LocCurv}(q)$ can be used to characterize almost everywhere the Gaussian curvature with an approach which differs from the one followed by Machigashira, \cite{Machi}, who exploited in his definition the upper excess of geodesic triangles.

\section{Convergence of curvature measures}\label{sect_conv}
Let $\mathcal K(\H^n)$ be the set of compact convex sets in $\H^n$ with nonempty interior. For any $K\in\mathcal K(\H^n)$ and $\rho>0$ define the set 
\[
K_\rho:=\{x\in \H^n\ :\ \dih(x,K)\leq\rho\}.
\]
The maps $f_K:\H^n\setminus K\to \partial K$ and $F_K:\H^n\setminus K\to T_{\partial K}\H^n$ are defined by the relations 
\[ \dih(f_K(x),x)=\dih(x,K)\qquad\text{and}\qquad x=\exp^{\H^n}_{f_K(x)}(d(K,x)F_K(x)),
\]
and are well-defined since $K$ is convex. For $\beta\subset\H^n$, define also 
\[
M_\rho(K,\beta)=f_K^{-1}(\beta\cap \partial K)\cap(K_\rho\setminus K).\] 

Following \cite{Ko-GeoDed}, given a convex set $K$ and $\rho>0$, let us  define a Radon measure $\mu_\rho$ on the Borel $\sigma$-algebra of the hyperbolic space $\mathcal{B}(\H^n)$ by
$$
\mu_\rho(K,\beta)=\vol_{\H^n}(M_\rho(K,\beta)).
$$
Set 
$$\ell_{n+1-r}(t):=\int_0^t\sinh^{n-r}(x)\cosh^r(x)dx,\quad r=0,\dots,n.$$
Then P. Kohlmann proved the following Steiner-type formula.
\begin{theorem}[Theorem 2.7 in  \cite{Ko-GeoDed}]\label{th_kohl}
There exists a family $\{\Phi_r(K,\cdot)\}_{r=0}^{n}$ of measures on $\mathcal{B}(\H^n)$ such that
\begin{equation}\label{steiner}
\mu_\rho(K,\beta)=\sum_{r=0}^{n}\ell_{n+1-r}(\rho)\Phi_r(K,\beta),\quad\forall \beta\in\mathcal B(\H^n).
\end{equation}
\end{theorem}

To our purpose, it is worth that, whenever $\eta=\partial K\cap \beta$ is a $C^3$ surface, the Borel measures $\Phi_r(K,\cdot)$ have the nice expression
\begin{equation}\label{eq_regphi}
\Phi_r(K,\beta)=\binom{n}{r}\int_{\eta}H^K_{n-r}(q) d\sigma_{\partial K}(q),
\end{equation}
where $\sigma_{\partial K}$ is the surface measure on $\partial K$ induced by $\vol_{\H^n}$. Hence Theorem \ref{th_kohl} recovers the classical regular Steiner formula in $\H^n$, \cite{Al-BAMS}. Here, $H^K_k(q)$ is the $k$-th symmetric function of the principal curvatures of $\partial K$ at $q$. In particular $H^K_0=1$,  $H^K_1$ is the mean curvature of $\partial K\subset \H^n$ and $H^K_n$ its Gaussian curvature.

\begin{remark}{\rm
For completeness, we recall that Kohlmann proved also for non-regular convex sets a more general integral representation for the Borel measures $\Phi_r(K,\cdot)$. This is given by
\[\Phi_r(K,\beta)=\binom{n}{r}\int_{\Pi^{-1}(\beta)\cap \mathcal N_K}g(v)\tilde H_{n-r}(v) d\mathcal H^{n}(v),\quad r=0,\dots,n,\]
where
 \begin{itemize}
\item $\mathcal N_K\subset T\H^n$ is the unitary normal bundle along $\partial K$, 
\item $\Pi:T\H^n\to \H^n$ is the standard projection on the base point,
\item the product $g(v)\tilde H_r(v)$ is defined for $\mathcal H^n$-a.e. $v\in \mathcal N_K$ by the algebraic expressions
\[
\tilde H_r(v) := {\binom nr}^{-1}\sum_{1\leq i_1<\cdots< i_r\leq n}\prod_{j=1}^r \tilde k_{i_j}(v),\]
and 
\[
g(v):= \infty^{-s}\prod_{i=s+1}^n\sqrt{\tilde k_i^2(v)+1},
\]
\item the (almost everywhere defined) generalized principal curvatures $\tilde k_i(v)$ can be roughly seen as limits as $\epsilon\to 0$ of the principal curvatures of the parallel sets $\partial K_\epsilon$ at point $\exp_{\Pi(v)}(\epsilon v)$, and $s=s(v)$ is such that $\tilde k_i(v)=\infty$ if and only if $i\leq s$. 
\end{itemize}
The pretty involved rigorous definitions and further details can be found in \cite[Sections 1 and 2]{Ko-GeoDed}. }\end{remark}

As in the classical setting of convex bodies in Euclidean space, \cite{Schn}, the curvature measures introduced by Kohlmann are solid enough to be weak continuous with respect to the topology induced on $\mathcal K(\H^n)$ by the Hausdorff distance $\dist_{\mathcal H}(K,L):= \inf\{\lambda>0\ : K\subset B_\lambda(L)\text{ and }L\subset B_\lambda(K)\}$, with $B_\lambda(L):=\{q\in\H^n\ :\ \dih(q,L)<\lambda\}$.
\begin{theorem}\label{th_convergingcurv}
Let $\{K_j\}_{j=1}^\infty\subset\mathcal K(\H^n)$ be a sequence of convex sets such that $K_j\to K$ as $j\to\infty$ in the Hausdorff topology. Then for every $r=0,\dots, n$ we have 
\[\Phi_r(K_j,\cdot)\to\Phi_r(K,\cdot) \]
as $j\to\infty$, weakly in the sense of measure.
\end{theorem}
According to the Steiner formula \eqref{steiner}, this latter theorem is a direct consequence of the following

\begin{proposition}\label{prop_continuity}
Let $\{K_j\}_{j=1}^\infty\subset\mathcal K(\H^n)$ be a sequence of convex sets such that $K_j\to K$ as $j\to\infty$ in the Hausdorff topology. Then, for every $\rho>0$,
\[\mu_\rho(K_j,\cdot)\to\mu_\rho(K,\cdot) \]
as $j\to\infty$, weakly in the sense of measure.
\end{proposition}
\begin{proof}
We start proving the following lemma

\begin{lemma}\label{lem_projection}
Let $K,L\subset \H^n$ be closed convex sets and $x\in \H^n$.\\ 
Let $d=\min\{\dih(x,f_K(x));\dih(x,f_L(x))\}$ and $\delta=\dist_{\mathcal H}(K,L)$. Then 
\begin{equation}\label{eq_proj}\dih(f_K(x),f_L(x))\leq \arccosh \left(\cosh \delta \frac{\cosh(d+\delta)}{\cosh(d-\delta)}\right),\end{equation}
whenever $\delta<d$. In particular
\[
\dih(f_K(x),f_L(x))\leq 2(\tanh d)^{1/2}\delta^{1/2}+O(\delta),\quad \text{as }\delta\to 0.\]
\end{lemma}
\begin{proof}
We adapt to the hyperbolic setting the corresponding Euclidean proof \cite[Lemma 1.8.9]{Schn}. The inequality is trivial if $x\in K\cap L$. If $x\in K\setminus L$, then 
\[
\dih(f_K(x),f_L(x))=\dih(x,f_L(x))=\dih(x,L)\leq \delta,
\]
so that \eqref{eq_proj} is satisfied. So let us suppose $x\not\in (K\cup L)$ and $d=\dih(x,f_K(x))$. By definition of $\delta$, we have $L\cap B_\delta(f_K(x))\neq \emptyset$, which implies $\dih(L,x)\leq d+\delta$ and $f_L(x)\in B_{d+\delta}(x)$.
Let $\gamma_1$ be the minimizing geodesics connecting  $f_L(x)$ and $f_K(f_L(x))$. It holds $L(\gamma_1)=\dih(f_L(x),K)\leq \delta$. Let $P$ be the totally geodesic hyperbolic hyperplane which is a suport plane to $K$ at $f_K(x)$ and is orthogonal to $F_K(x)$. 

Suppose first that $f_L(x)$ and $f_K(x)$ are in the same half-space with respect ot $P$. Then the angle $\alpha:=\widehat{xf_K(x)f_L(x)}\geq\pi/2$ so that by the hyperbolic cosine law
$$\cosh( \dih(f_K(x),f_L(x)))\cosh( \dih(f_K(x),x))\leq \cosh( \dih(x,f_L(x)))\leq\cosh(d+\delta).$$ Hence
$$
\cosh( \dih(f_K(x),f_L(x)))\leq \frac{\cosh(d+\delta)}{\cosh( d)}\leq \cosh(\delta)\frac{\cosh(d+\delta)}{\cosh( d-\delta)}.
$$

Suppose now that $f_L(x)$ and $f_K(x)$ are in opposite half-spaces with respect to $P$. Let $b$ be the intersection point of $\gamma_1$ with $P$, so that $\dih(b,f_L(x))\leq\delta$. Let $E$ be the point obtained by projecting $f_L(x)$ onto the (convex) geodesic $\gamma_2$ of $\H^n$ which pass through $x$ and $f_K(x)$. Note that $f_K(x)$ is the projected of $B$ onto $\gamma_2$. By the Hyperbolic Busemann-Feller Lemma, \cite[Proposition II.2.4]{BH}, projection onto a complete convex set in $\H^n$ is distance decreasing. In particular $\dih(E,f_K(x))\leq \dih(b,f_L(x))\leq\delta<d$, and $E$ is between $x$ and $f_K(x)$ on $\gamma_2$. Then $\dih(x,E)=\dih(x,f_K(x))-\dih(E,f_K(x))\geq d-\delta$. Pythagoras'
theorem for hyperbolic triangles gives 
\[\cosh(\dih(x,f_L(x)))=\cosh(\dih(E,f_L(x)))\cosh(\dih(x,E))\] and 
\[\cosh(\dih(f_K(x),f_L(x)))=\cosh(\dih(E,f_L(x)))\cosh(\dih(f_K(x),E)). \]
All together implies 
\[\cosh(\dih(f_K(x),f_L(x)))\leq \cosh(\delta)\frac{\cosh(d+\delta)}{\cosh( d-\delta)} \]
as desired.
\end{proof}
We come back to the proof of Proposition \ref{prop_continuity}, inspired by \cite[Theorem 4.1.1]{Schn}. Let $\{K_j\}_{j=1}^\infty\subset\mathcal K(\H^n)$ be a sequence of convex sets such that $K_j\to K$ as $j\to\infty$ in the Hausdorff topology. Let $U\subset \H^n$ be an open set and let $q
\in M_\rho(K,U)\setminus \partial K_\rho$, so that $f_K(q)\in U$. According to Lemma \ref{lem_projection}, for $j$ large enough $f_{K_j}(q)\in U$ and $\dih(q,K_j)\leq\dih(q,K)+ O((\dist_{\mathcal{H}}(K,K_j))^{1/2})<\rho$. Then for $j$ large enough $q\in M_\rho(K_j,U)$, so that $M_\rho(K,U)\setminus \partial K_\rho\subset \liminf_{j\to\infty}M_\rho(K_j,U)$. By Fatou's lemma
\begin{align}\label{Fatou}
\mu_\rho(K,U)=\vol_{\H^n}\left(M_\rho(K,U)\setminus \partial K_\rho\right) \leq \liminf_{j\to\infty}\mu_\rho(K_j,U).
\end{align}
On the other hand, let $\delta_j=\dist_{\mathcal H}(K,K_j)$. Then, reasoning as above, $M_\rho(K_j,\H^n)\subset M_{\rho+\epsilon_j}(K,\H^n)=M_{\epsilon_j}(M_{\rho}(K,\H^n),\H^n)$, where $\epsilon_j\to0$ as $j\to\infty$. Since $M_{\rho}(K,\H^n)\in \mathcal K(\H^n)$, by Theorem \ref{th_kohl}
\[
\mu_{\epsilon_j}(M_{\rho}(K,\H^n),\H^n)=\mu_{\rho}(K,\H^n)+ O(\epsilon_j),\]
so that 
\[\limsup_{j\to\infty} \mu_\rho(K_j,\H^n) \leq \mu_{\rho}(K,\H^n).
\]
This latter, together with \eqref{Fatou} and standard measure theory, \cite[p. 196]{Ash}, concludes the proof.
\end{proof}


\section{Boundary structure}\label{sect_bdy} In this section we study the regularity of the boundary of a convex set $K\in\mathcal K(\H^n)$. As in the Euclidean setting, the main tool is the first and second differentiability property of real convex functions at almost every point. 

Up to an isometry of the ambient space, the boundary of $K$ is locally the graph of a (horo)convex function. Namely, let $q_0\in \partial K$. Consider the Poincar\'e half-space model $\H^n$ with its global coordinates system $(\mathbf{x},z)_E\in \R^{n-1}\times(0,+\infty)$ and metric $g_{\H^n}=z^{-2}(d\mathbf{x}^2+dz^2)$ of constant curvature $-1$, where $d\mathbf{x}^2=\sum_{i=1}^{n-1}dx_i^2$ is the Euclidean metric of $\R^{n-1}$.  Up to an isometry of $\H^n$ we can suppose that $q_0=(\mathbf 0,1)_E$ and, since $K$ has nonempty interior, that $\{(\mathbf 0,1+t)_E\ :\ 0<t\leq t_0\}\subset \mathrm{int}(K)$ for $t_0$ small enough.
Consider the horosphere $\mathfrak{h}=\{(\mathbf x,1)_E~:~\,\mathbf x\in\R^{n-1}\}$ and note that $\mathbf x\mapsto (\mathbf x,1)_E$ is an isometry from $\R^{n-1}$ to $\mathfrak{h}$. As in \cite{FIV}, let us introduce a coordinate system $(\mathbf{\xi},\zeta)_{\mathfrak{h}}$ on $\H^n$, called horospherical coordinates, defined by $(\mathbf{\xi},\zeta)_{\mathfrak{h}}=(\mathbf{\xi},e^{-\zeta})_{E}$ so that $\mathfrak{h}=\{(\mathbf \xi,0)_{\mathfrak{h}}:\xi\in\R^{n-1}\}\cong\R^{n-1}$ and $\zeta$ is the signed hyperbolic distance from $(\mathbf \xi,e^{-\zeta})_E$ to the horosphere $\mathfrak{h}$, which is realized by the the "vertical" unit speed geodesic $t\in[0,\zeta]\mapsto(\mathbf x,t)_{\mathfrak{h}}$. We have chosen the sign of $\zeta$ in such a way that $(\mathbf 0,\zeta)_\mathfrak{h}\not\in K(u)$ for $\zeta>0$ small enough. Note also that the hyperbolic metric in horospherical coordinates at $(\mb\xi,\zeta)_\mathfrak{h}$ is given by \begin{equation}\label{horometric}
e^{2\zeta} d\mb\xi^2+d\zeta^2.
\end{equation}

Thanks to the relative position of $K$ and $\mathfrak{h}$, there exists $R>0$ and $u:\mathbb B_R\subset \R^{n-1}\to(0,+\infty)$ such that a neighborhood of $x_0$ in $\partial K$ coincides with the graph 
\begin{equation}\label{graph}\mathfrak g_\mathfrak{h}(u):=\{(\mathbf x,u(\mathbf x))_\mathfrak{h}\in\H^{n}\ :\ \mathbf x\in\B_R\}\end{equation}
of $u$ on $\B_R$. Here and on, $\B_R(\mbx)\subset \R^{n-1}$ denotes the Euclidan ball of radius $R$ centered at $\mbx$, and we just write $\B_R$ whenever $\mbx=\mbz$. Since $K$ is convex, $u$ is \textsl{horoconvex}\footnote{The definition of horoconvex function could be misleading, since there is in the literature a well studied concept of horoconvex sets (also called h-convex or horospherically convex) which denote sets in $\mathcal K(\H^n)$ having at each point a supporting horosphere. This is not directly related to horoconvex functions.}, which means that the function $h_u:\mathbb B_R\to \R$ defined by 
$$h_u(\mathbf x)= e^{-2u(\mathbf x)}+|\mathbf x|^2$$ is convex. By the way, this guarantees that both $u$ and $h_u$ are Lipschitz. Conversely, whenever $v:\mathbb B_R\subset \R^{n-1}\to(0,+\infty)$ is horoconvex, then the graph $\mathfrak g_\mathfrak{h}(v)$ of $v$, defined as in \eqref{graph}, is an open set of the boundary of a convex set $K=K(v)\in\mathcal K(\H^3)$.
Note that, in order to define the whole convex set $K(v)$ containing $\mathfrak g_\mathfrak{h}(v)$ (and thus  rigorously justify some subsequent arguments), one can for instance choose 
\begin{equation}
\label{def_Kv}
K(v)=\{(\mathbf x, z)\in\H^n\ :\ \mathbf x\in \mathbb B_R\text{ and }v(\mathbf x)\leq z \leq \zeta\},\end{equation}
with $\zeta>\max_{\B_R}v$ large enough to be fixed later. 
The equivalence between convexity of sets and horoconvexity of functions is proved for instance in \cite{GSS-JGA} for smooth convex sets and in \cite[Proposition 2.6]{FIV} for non-necessarily smooth convex sets in $\H^3$. The generalization to any dimension is straightforward. 

Since $h_u$ is convex, according to a celebrated theorem by Busemann-Feller and Alexandrov, \cite{BF,Al-Lenin}, the set of points $$\mathcal D_u''=\{\mathbf x\in\B_R\ :\ h_u\text{ is 2 times differentiable at }\mathbf x\}$$
has full Lebesgue measure in $\B_R$. Namely, for $\mathbf x\in\mathcal D_u''$, the subgradient $\partial_{\mathbf{x}}h_u$ of $h_u$ at $\mbx$, defined by
$$\partial_{\mathbf{x}}h_u:=\{\mathbf v\in\R^{n-1}\ :\ \forall \mby \in \B_R,\ h_u(\mby)\geq h_u(\mbx)+\left\langle \mathbf v,\mby-\mbx\right\rangle\},
$$ contains a unique element $\nabla_{\mathbf{x}}h_u$ and there exists a symmetric matrix $\nabla^2_{\mathbf x}h_u$ such that
\[
h_u(\mathbf y)=h_u(\mathbf x)+\left\langle\nabla_{\mathbf x}h_u,\mathbf y-\mathbf x\right\rangle+\frac12 \left\langle \nabla^2_{\mathbf x}h_u(\mathbf y-\mathbf x),\mathbf y-\mathbf x\right\rangle + o(|\mathbf y-\mathbf x|^2).
\]
An easy computation shows then that also $u$ is two times differentiable at $\mathbf x$, i.e.
\begin{equation}\label{2ndord}
u(\mathbf y)=\bar u(\mathbf y)+  o(|\mathbf y-\mathbf x|^2),
\end{equation}
where 
\begin{equation}\label{def_ubar}
\bar u(\mathbf y)=u(\mathbf x)+\left\langle\nabla_{\mathbf x}u,\mathbf y-\mathbf x\right\rangle+\frac12 \left\langle\nabla^2_{\mathbf x}u (\mathbf y-\mathbf x),\mathbf y-\mathbf x\right\rangle,
\end{equation}
with $\nabla_{\mathbf x}u=-\frac12 e^{2u(\mathbf x)}(\nabla_{\mathbf x}h_u-2\mathbf x)$ and 
\[\nabla^2_{\mathbf x}u=-\frac{1}{2}e^{2u(\mathbf x)}\nabla^2_{\mathbf x}h_u+e^{2u(\mathbf x)} I_{\R^{n-1}}
+2\nabla_{\mathbf x}u\otimes \nabla_{\mathbf x}u. \]

In what follows we want to prove that almost every point of the boundary is normal, that is that the second differentiability property at a.e. point of the horoconvex representation is preserved under changes of the reference horosphere. 

\textsl{Smooth} points are those boundary  points admitting a unique support hyperbolic (hyper)plane. The smoothness of a point $q\in \partial K$ clearly do not depend on the chosen horosphere $\mathfrak{h}$, since  it is equivalent to the fact that the smooth convex surface $\mathfrak g_\mathfrak{h}(\bar u)$ is tangent to $\partial K$ at $q$. In particular, by Rademacher's theorem a.e. boundary point is smooth. So let $q\in \partial K$ be a smooth point. There is a unique exterior unit normal $\nu_q$ to $\partial K$ at $q$, and hence a unique interior tangent horosphere $\mathfrak{h}_q$, that is the horosphere orthogonal to $\nu_q$ and such that $\exp_q(-t\nu_q)\in \mathrm{int}(K)$ for $t>0$ small enough. Since $\partial K$ is tangent to $\mathfrak{h}_q$ at $q$, it admits a (unique) local parametrization via a horoconvex function $u_q:\B_\epsilon\subset \mathfrak{h}_q\to \R$, for some small $\epsilon>0$, with 
\begin{equation}\label{1stordernormal}
u_q(\mathbf 0)=0, \quad\text{ and }\quad \nabla_{\mathbf 0}u_q=0.
\end{equation} 
We say that $q$ is a \textsl{normal} boundary point if $u_q$ is two times differentiable at $q$. By the way, note that at normal points, $\nabla u_q$ is differentiable, in the sense that
\begin{equation}\label{mignot}
\sup_{\mbx\in \B_\epsilon}\sup_{V\in\partial_\mbx u_q} |V - \nabla_\mbx \bar u_q |= \sup_{\mbx\in \B_\epsilon}\sup_{V\in\partial_\mbx u_q} |V - \nabla^2_\mbz u_q (\mbx) |=o(\epsilon),
\end{equation}
see \cite[Corollary 2.4]{Rock}. By properties \eqref{1stordernormal}, here the expression for $\bar u_q$ simplifies as \[\bar u_q(\mbx)=\frac{1}{2}\nabla^2_\mbz u_q (\mbx,\mbx).\]

\begin{proposition}\label{prop_normal}
Let $K\in \mathcal K (\H^n)$. Then $\mathcal H^{n-1}$-a.e. point in $\partial K$ is normal.
\end{proposition}
For the corresponding result in 
the Euclidean setting, see \cite{BF,Al-Lenin} or \cite[Theorem 2.5.5]{Schn}.

\begin{proof}[Proof (of Proposition \ref{prop_normal}).]
Let $\mathfrak{h}$ be a horosphere of $\H^n$ and $u:\B_R\subset \mathfrak{h}\to\R$ a horoconvex function locally  parametrizing an open set of $\partial K$. Acccording to what said above, the set $\mathcal{D}''_u$ of two times differentiable points has full Lebesgue measure in  $\B_R$. Since $u$ is Lipschitz we deduce that $\mbx\in\B_R\mapsto (\mbx,u(\mbx))_{\mathfrak{h}}\in\partial K\subset \H^n$ is biLipschitz onto its image, and thus $\{(\mathbf x,u(\mathbf x))_{\mathfrak{h}}\ :\ \mathbf x\in \mathcal{D}''_u\}$ has full Hausdorff measure in $\mathfrak g_\mathfrak{h}(u|_{\B_R})\subset\partial K$. We are going to show that every such a point is normal. 

Fix $\mathbf x_0\in \mathcal D''_u$. 
Let $\mathfrak{h}_1\stackrel{iso}{\cong}\R^{n-1}$ be the horosphere $\{(\mathbf y,\zeta)_\mathfrak{h}\in \H^3\ :\ \zeta\equiv u( \mathbf x_0)\}$, and let $\mathfrak{h}_2\stackrel{iso}{\cong}\R^{n-1}$ be a horosphere of $\H^n$ tangent to the graph of $u$ at $q:=(\mathbf x_0,u(\mathbf x_0))_\mathfrak{h}$.
As above, we introduce two horospherical coordinates systems $(\mathbf{y},z)_{\mathfrak{h}_i}$, $i=1,2$, on $\H^n$ relative to the horospheres $\mathfrak{h}_1$ and  $\mathfrak{h}_2$, and fix the Euclidean coordinates on $\mathfrak{h}_i$ in such a way that $q=(\mathbf 0,0)_{\mathfrak{h}_i}$ for $i=1,2$.
Since both graphs $\mathfrak g_{\mathfrak{h}_1} (u|_{\B_R})$ and $\mathfrak g_{\mathfrak{h}_1} (\bar u|_{\B_R})$ are tangent to $\mathfrak{h}_2$ at $\mathbf 0$, by the implicit function theorem they have respective local representations as graphs $\mathfrak g_{\mathfrak{h}_2} (u')$ and $\mathfrak g_{\mathfrak{h}_2} (\bar u')$ of functions $u',\bar u':\Omega \to \R$ defined on an open neighborhood $\Omega$ of $\mathbf 0\in \mathfrak{h}_2$. 
Both $u'$ and $\bar u'$ are differentiable at $\mbz$ and $\nabla_\mbz u'=\nabla_\mbz \bar u'=0$. Finally, since $\bar u\in C^\infty$, also  $\bar u'\in C^\infty$.

We will use repeatedly the following
\begin{lemma}\label{lem_dist} Let $\mathfrak{h}\subset \H^{n}$ be a fixed horosphere, then \begin{equation}\label{ineq_proj}
|\mb a-\mb b|\leq e^{\max\{|\zeta_a|,|\zeta_b|\}}\dih((\mb a,\zeta_a)_\mathfrak{h},(\mb b,\zeta_b)_\mathfrak{h}).
\end{equation}
\end{lemma}
\begin{proof}
Since the projection on a convex set is a contraction, \cite[Proposition II.2.4]{BH}, by definition of the induced metric also the projection on the boundary of the convex set is a contraction. The inequality \eqref{ineq_proj} follows from projecting on the convex horosphere $\{(\mby,e^{-\max\{|\zeta_a|,|\zeta_b|\}})\ :\ \mby\in\R^{n-1}\}$, together with a simple explicit calculation.
\end{proof}
Let $(\mb \xi, 0)_{\mathfrak{h}_2}\in \mathfrak{h}_2$ such that $|\mb \xi|<\epsilon$. We need to prove that $|u'(\mb\xi)-\bar u'(\mb\xi)|=o(\epsilon^2)$. Define points $A,\bar A\in \H^n$ and $\mby_A,\bar \mby_A\in \mathfrak{h}_1$ by $A:=(\mb \xi, u'(\mb\xi))_{\mathfrak{h}_2}=:(\mby_A, u(\mby_A))_{\mathfrak{h}_1}$ and $\bar A:=(\mb \xi, \bar u'(\mb\xi))_{\mathfrak{h}_2}=:(\bar \mby_A, \bar u(\bar \mby_A))_{\mathfrak{h}_1}$. This is clearly possible for $\epsilon$ small enough. According to \eqref{ineq_proj}, for $\epsilon \ll 1$, 
\begin{align*}
|\mby_A|&\leq 2  \dih( A,q)\\&\leq2  \dih( A,(\mb\xi,0)_{\mathfrak{h}_2})+2  \dih( (\mb\xi,0)_{\mathfrak{h}_2},q)\\&= 2u'(\mb\xi)+ 2\arcsinh (|\mb \xi|/2) \\&= o(\epsilon).
\end{align*}
The second differentiability of $u$ at $\mbz$ gives then
$|u(\mby_A)-\bar u(\mby_A)|=o(\epsilon^2)$,
i.e. \begin{equation}\label{distAB}
\dih(A,B)=o(\epsilon^2),
\end{equation}
where we have posed $B:=(\mby_A,\bar u(\mby_A))_{\mathfrak{h}_1}=:(\mb\xi_B,\bar u'(\mb\xi_B))_{\mathfrak{h}_2}$, for some (unique) $\mb\xi_B\in \mathfrak{h}_2$. Applying again Lemma \ref{lem_dist} with $\mathfrak{h}=\mathfrak{h}_2$ we get 
\[|\mb\xi-\mb\xi_B|\leq 2 \dih(A,B)=o(\epsilon^2).\]
Since $\bar u'$ is smooth, 
\[|\bar u'(\mb\xi) -\bar u'(\mb\xi_B)|\leq L  |\mb\xi-\mb\xi_B|=o(\epsilon^2),\]
$L$ being the local Lipschitz constant of $\bar u'$ in a neighborhood of $\mbz$.
Accordingly,
\begin{align*}
\dih(B,\bar A)&= \dih((\mb\xi_B,\bar u'(\mb\xi_B))_{\mathfrak{h}_2},(\mb \xi, \bar u'(\mb\xi))_{\mathfrak{h}_2})\\
&\leq\dih((\mb\xi_B,\bar u'(\mb\xi_B))_{\mathfrak{h}_2},(\mb \xi_B, \bar u'(\mb\xi))_{\mathfrak{h}_2})+\dih((\mb\xi_B,\bar u'(\mb\xi))_{\mathfrak{h}_2},(\mb \xi, \bar u'(\mb\xi))_{\mathfrak{h}_2})\\
&\leq |\bar u'(\mb\xi_B)- \bar u'(\mb\xi)|+2\arcsinh\left( \frac{e^{\bar u'(\mb\xi)}|\mb\xi_B-\mb\xi|}{2}\right)\\&=o(\epsilon^2).
\end{align*}
Finally, this latter and \eqref{distAB} give
\begin{align*}
|u'(\mb\xi)-\bar u'(\mb\xi)|&=\dih(A,\bar A)\\&\leq \dih(A,B)+\dih(B,\bar A)\leq o(\epsilon^2).
\end{align*}
\end{proof}

\section{biLipschitz approximation around normal points}\label{sect_biLip}

In this section we will show that around a.e. (normal) point of $\partial K$, $\epsilon$-balls are $(1+o(\epsilon^2))$-biLipschitz equivalent to $\epsilon$-balls on a smooth manifold. In the case of 2 dimensional surfaces, one can choose the smooth manifold to be of constant curvature. This property will be subsequently exploited to give a characterization of the regular part of the curvature measure for convex surfaces; see Section \ref{sect_reg}.

\begin{proposition}\label{pr_biLip}
Let $K\in \mathcal K(\H^n)$ and let $q\in\partial K$ be a normal point. For $\epsilon\ll 1$, the metric ball  $B_\epsilon(q)\subset\partial K$ is biLipschitz equivalent to an open set $U_\epsilon\subset\partial \bar K$ contained in the boundary of a smooth convex set $\bar K\in\mathcal K(\H^n)$, with Lipschitz constants smaller than $(1+o(\epsilon^2))$ as $\epsilon\to 0$.
\end{proposition}
\begin{corollary}\label{coro_biLip}
Let $K\in \mathcal K(\H^3)$ and let $q\in\partial K$ be a normal point. For $\epsilon\ll 1$, the metric ball  $B_\epsilon(q)\subset\partial K$ is biLipschitz equivalent to an open set $U_\epsilon\subset S_k$ contained in a smooth surface $S_k$ of constant sectional curvature $k=k(q)$, with Lipschitz constants smaller than $(1+o(\epsilon^2))$ as $\epsilon\to 0$.
\end{corollary}

We call the (unique) real number $k$ given by Corollary \ref{coro_biLip} the local curvature of $\partial K$ at $q$, and we denote it by $\mathrm{LocCurv}_{\partial K}(q)$.

\begin{proof}[Proof (of Proposition \ref{pr_biLip}).]
Let $q\in \partial K$ be a normal point and let $u=u_q:\B_R\subset \mathfrak{h}\to\R$ be the corresponding local parametrization as the graph of a horoconvex function over a horosphere $\mathfrak{h}=\mathfrak{h}_q$. In particular, for $\epsilon$ close to $0$, the asymptotic relation \eqref{mignot} holds. We let $\bar K_q$ be the smooth hypersurface of $\H^n$ given by the graph $\mathfrak{g}_{\mathfrak{h}_q}(\bar u)$. 
For any horoconvex function $v:\B_R\subset \mathfrak{h}\to (0,+\infty)$ we define $\hat d_v$ the metric induced on $\partial K(v)$ by $\H^3$ and $d_v(\mathbf x,\mathbf y)=\hat d_v((\mathbf x,v(\mathbf x))_\mathfrak{h},(\mathbf y,v(\mathbf y))_\mathfrak{h})$ for any $\mathbf x,\mathbf y\in \B_R$.

Fix $\epsilon>0$ and let $\mb p,\mb s\in\B_R$ such that $d_u(\mb p,\mbz)<\epsilon$ and $d_u(\mb s,\mbz)<\epsilon$.
Let $\gamma:[0,L]\to \mathfrak{h}$ be a Lipschitz curve with $\gamma(0)=\mb p$ and $\gamma(L)=\mb s$ and with a.e. unit speed with respect to the Euclidean metric of $\mathfrak{h}$. Denote by $\hat\gamma(t):= (\gamma(t),u(\gamma(t))_{\mathfrak{h}}$ and $\bar\gamma(t):= (\gamma(t),\bar u(\gamma(t)))_{\mathfrak{h}}$ the corresponding lifted Lipschitz curves in $\partial K$ and $\partial \bar K$ respectively.
We can choose $\gamma$ in such a way that $\nabla_{\gamma(t)}u$ exists for a.e. $t\in[0,L]$ and the length $\hat L$ of $\hat \gamma:[0,L]\to \partial K$ is less then $(1+\e^3)d_u(\mb p,\mb s)$. 
Since $d_u(\mb p,\mb s)<2\epsilon$ and the identity map $(\B_R,d_{\R^{n-1}})\to(\B_R,d_u)$ is biLipschitz, there exists a constant $\alpha$ independent of $\mb p$, $\mb s$ such that $\gamma([0,L])\subset \B_{\alpha\epsilon}$ and $L\leq\alpha d_u(\mb p,\mb s)$.
We can compute
\begin{align*}
\int_{0}^L |\dot{\bar\gamma}(t)|_{\partial K}dt
= \int_{0}^L \left(e^{2\bar u (\gamma(t))}+ \left\langle\nabla_{\gamma(t)}\bar u,\dot\gamma(t)\right\rangle^2\right)^{1/2}dt
\end{align*}
and
\begin{align*}
\int_{0}^L |\dot{\hat\gamma}(t))|_{\partial K}dt
= \int_{0}^L \left(e^{2u (\gamma(t))}+ \left\langle\nabla_{\gamma(t)} u,\dot\gamma(t)\right\rangle^2\right)^{1/2}dt.
\end{align*}
Note that
\begin{align*}
e^{2\bar u (\gamma(t))}=e^{2 u (\gamma(t))+o(\epsilon^2)}=e^{2 u (\gamma(t))}+o(\epsilon^2),
\end{align*}
and, using also \eqref{mignot},
\begin{align*}
\left\langle\nabla_{\gamma(t)} u,\dot\gamma(t)\right\rangle^2
-\left\langle\nabla_{\gamma(t)} \bar u,\dot\gamma(t)\right\rangle^2
&=\left\langle\nabla_{\gamma(t)} u-\nabla_{\gamma(t)} \bar u,\dot\gamma(t)\right\rangle\left\langle\nabla_{\gamma(t)} u+\nabla_{\gamma(t)} \bar u,\dot\gamma(t)\right\rangle\\
&\leq |\nabla_{\gamma(t)} u-\nabla_{\gamma(t)} \bar u||\nabla_{\gamma(t)} u+\nabla_{\gamma(t)} \bar u|\\
&=o(\epsilon)O(\epsilon)=o(\epsilon^2).
\end{align*}
Putting this all together gives
\begin{align*}
\int_{0}^L |\dot{\bar\gamma}(t))|_{\partial K}dt
&= \int_{0}^L \left(e^{2\bar u (\gamma(t))}+ \left\langle\nabla_{\gamma(t)} \bar u,\dot\gamma(t)\right\rangle^2\right)^{1/2}dt\\
&= \int_{0}^L \left(e^{2 u (\gamma(t))}+ \left\langle\nabla_{\gamma(t)} u,\dot\gamma(t)\right\rangle^2+o(\epsilon^2)\right)^{1/2}dt\\
&= \int_{0}^L \left\{\left(e^{2 u (\gamma(t))}+ \left\langle\nabla_{\gamma(t)} u,\dot\gamma(t)\right\rangle^2\right)^{1/2}+o(\epsilon^2)\right\}dt\\
&=\hat L +Lo(\epsilon^2)\\
&\leq \hat L+\alpha d_u(\mb p,\mb s)o(\epsilon^2)\\
&\leq d_{u}(\mb p,\mb s)(1+ o(\epsilon ^2)),
\end{align*}
from which
\[
d_{\bar u}(\mb p,\mb s)\leq d_{u}(\mb p,\mb s)(1+ o(\epsilon ^2)).
\]
Interchanging the role of $u$ and $\bar u$ we get the converse relation \[
d_{u}(\mb p,\mb s)\leq d_{\bar u}(\mb p,\mb s)(1+ o(\epsilon ^2)).
\]
The proposition is thus proved with $U_\epsilon=\mathfrak{g}_\mathfrak{h}(\bar u|_{V_\epsilon})$, where $V_\epsilon:=\{\mbx\in \mathfrak{h}\ :\ d_u(\mbx,\mbz)<\epsilon\}$.
\end{proof}
\begin{proof}[Proof (of Corollary \ref{coro_biLip}).]
Let $\partial\bar K$ be the smooth surface containing $q$ given by Proposition \ref{pr_biLip}. Let $S_k$ be a smooth surface of constant sectional curvature $k$ equal to the sectional curvature of $\partial \bar K$  at the point $q$, and let $p$ be any reference point in $S_k$. According to the asymptotic development of the distance function in geodesic normal coordinates, (see \cite{Gallot,Br-CQG} or \cite[Lemma 13]{Ve}), distances in $B_\epsilon\subset \partial\bar K$ are equivalent to distances in the geodesic ball $B^{S_k}_\epsilon(p)$ of $S_k$ up to the third order as $\epsilon\to 0$. Hence the corollary follows.
\end{proof}

\section{Generalized Gauss theorem}
In this section we focus on convex sets of $\H^3$.  In particular, we want to prove that the well-known Gauss theorem connecting the scalar curvature of a surface, its extrinsic Gaussian curvature and the curvature of the ambient space generalizes to non-smooth convex surfaces in $\H^3$. This will be done considering a smooth approximation of the convex surface (locally parametrized as a horoconvex function) and showing that Gauss equation passes to the limit. As a corollary, we obtain also that Alexandrov's approach to extrinsic curvature of convex surfaces in $\H^3$ is in fact equivalent to the one by Kohlmann we use here.

First, recall that the boundary of a convex set $K\in \mathcal K (\H^3)$ has an intrinsic curvature measure $\omega_{\partial K}$ which generalizes to the boundary of non-smooth convex sets the integral of the usual sectional curvature in case of smooth surfaces.

In fact, given a convex set $K\in \mathcal K (\H^3)$, since its boundary has curvature in the sense of Alexandrov lower bounded by $-1$, the surface $\partial K$ endowed with the metric induced by $\H^3$ turns out to be a special case of surface of Bounded Integral Curvature (often shortened in \textsl{BIC surface} in the literature); see for instance \cite[p. 359]{BBI} and \cite{Machi} or \cite[Theorem 3.2]{AB}.  There exist several equivalent definitions of a BIC surface, \cite{Re,AZ} and \cite[Section 2]{Tr}. For instance, one can define it as a topological surface $S$ endowed with a metric $d$ obtained as a uniform limit of a sequence of distances $\{d_j\}$ induced by a sequence of smooth Riemannian metrics $\{g_j\}$ on $S$, with the property that the $\int_S|\sect_{g_j}|d\mathrm{Vol}_{g_j}$ is uniformly bounded along the sequence. The curvature measure $\omega_S$ is then the weak limit in the sense of measures of the sequence of measures $\alpha\mapsto \int_\alpha\sect_{g_j}d\mathrm{Vol}_{g_j}$.\\

%

Let $v:\B_R\subset \mathfrak{h}\to\R$ be any horoconvex function. For $k=0,1,2$, let us introduce Borel mesures $\hat H^v_k$ defined for every Borel set $\beta\subset \B_R$ by 
\begin{equation}\label{def_hatmu}
\left(\begin{array}{c}
n\\k
\end{array}\right) \hat H^v_k(\beta)=\Phi_{3-k}(K(v),\mathfrak{g}_\mathfrak{h}(v|_\beta)).
\end{equation} 
Since $v$ is Lipschitz, its graph $\mathfrak{g}_\mathfrak{h}(v|_{\beta})$ restricted to $\beta\in\mathcal B(\R^2)$ is a Borel set of $\H^3$, so that the measures $\hat H^v_k$ are well defined. According to \eqref{eq_regphi}, whenever $v$ is smooth,
\begin{equation*}
\hat H^v_k(\beta)=\int_{\beta}H^K_k((\mbx,v(\mbx))_{\mathfrak h})d\hat H^v_0(\mbx)
\end{equation*} 
is the $k$-th symetric function of the principal curvatures of $\partial K$ integrated with respect to the area measure of $\partial K$.

Moreover,  we introduce also the Borel measure $\hat\omega_v$ on $\mathbb B_R$, where for $\beta\in\mathcal B(\B_R)$, 
\[\hat\omega_v(\beta)=\omega_{\partial K}(\mathfrak{g}_{\mathfrak h}(v|_\beta))\] is the intrinsic curvature measure of the horograph of $v$ over $\beta$. 
As above, by the definition of the intrinsic curvature, whenever $v$ is smooth,
\begin{equation*}
\hat\omega_v(\beta)=\int_\beta \sect_{\partial K(v)}((\mbx,v(\mbx))_{\mathfrak h})d\hat H^v_0(\mbx).
\end{equation*} 

Now, fix a convex set $K=K(u)\in \mathcal K(\H^n)$. As in the previous sections, let $u:\B_R\subset \mathfrak{h}\to\R$ be a horoconvex function, defined on a horosphere $\mathfrak{h}$, whose graph gives a local parametrization of $\partial K$.
 There exists a sequence of $C^\infty$ horoconvex functions $u_j:\B_R\to(0,+\infty)$ such that $u_j$ converges to $u$ uniformly and $W^{1,2}_{loc}$ on $\B_R$ as $j\to\infty$. To prove this assertion, one can approximate $h(u)$ uniformly and $W^{1,2}_{loc}$ by a sequence $\{h(u_j)\}$ of convex functions $\B_R\to\R$ explicitly constructed via mollification, \cite[p. 123 and 238]{EG}. 
 
Let $K(u)$ and  $K(u_j)$ be defined as in \eqref{def_Kv} with $\zeta = \sup_{j} \max_{\B_R}u_j+1<+\infty$.
By \cite[Lemma 3.7]{FIV} we deduce
\begin{lemma}\label{lem_FIV}
Let $u:B_R\to (0,+\infty)$ be a horoconvex function. There exist constants $C,C'>0$ (depending on $u$) such that for any $0<r<R$ and any other horoconvex function $v:B_R\to (0,+\infty)$, if
$\sup_{\B_r}|u-v|<\epsilon$, then
$$\sup_{\mathbf x,\mathbf y\in\B_r}|d_u(\mathbf x,\mathbf y)-d_v(\mathbf x,\mathbf y)|\leq C\epsilon + C'\sqrt{\epsilon}r.$$
In particular $d_{u_j}$ converges uniformly to $d_{u}$ on $\B_R$ as $j\to\infty$.
\end{lemma}
Since the $u_j$'s are smooth, by \eqref{eq_regphi} the measures $\hat H^{u_j}_k$ are absolutely continuous with respect to the Lebesgue measures of $\R^2$, and the classical Gauss-Codazzi equation 
\begin{equation}\label{eq_GC_reg}
\forall q\in \mathfrak g_{\mathfrak h}(u_j),\qquad  H_2^{K(u_j)}(q)=\sect_{\partial K(v)}(q)-1
\end{equation}
implies
\[
\frac{1}{3}\Phi_{1}(K(u_j),\mathfrak{g}_\mathfrak{h}(u_j|_\beta))=\omega_{\partial K(u_j)}(\mathfrak{g}_{\mathfrak h}(u_j|_\beta))-\Phi_{3}(K(u_j),\mathfrak{g}_\mathfrak{h}(u_j|_\beta)).
\]
Projecting measures onto $\mathfrak h$, this latter gives
\begin{equation}\label{eq_GC}
\hat H^{u_j}_2(\beta)=\tilde \omega_{u_j}(\beta) -\hat H^{u_j}_0(\beta).
\end{equation}
Note that $\int_\beta \sect_{\partial K(u_j)}(\mathbf x)d\hat H^{u_j}_0(\mathbf x)= \hat\omega_{u_j}(\beta)$. Since $u_j\to u$ uniformly, by Lemma \ref{lem_FIV} and by definition of the curvature measure of a BIC surface, we get that 
\[\omega_{\partial K(u_j)}\to \omega_{\partial K(u)}\]
and equivalently
\[\hat\omega_{u_j}\to \hat \omega_{u}\]
 weakly in the sense of measure as $j\to\infty$. Lemma \ref{lem_FIV} implies also that $\dist_{\mathcal H}(K(u_j),K(u)) \to 0$ as $j\to\infty$. Thanks to Theorem \ref{th_convergingcurv}, we have then that \eqref{eq_GC} passes to the limit as $j\to\infty$ giving the following results, which can be seen as a version of Gauss-Codazzi equation \eqref{eq_GC_reg} for non-smooth convex surfaces.
\begin{theorem}[Generalized Gauss Theorem]\label{th_Gauss}
Let $K\in\mathcal{K}(\H^3)$ and $u:\mathfrak{h}\to\R$ a horoconvex local parametrization of $\partial K$, $\mathfrak{h}$ being a horosphere of $\H^n$. Then
\begin{equation}\label{eq_GClim}
\hat H^{u}_2+\hat H^{u}_0=\hat\omega_{u}
\end{equation}
and
\begin{equation}\label{eq_GClim2}
\frac{1}{3}\Phi_1(K,\alpha)+\Phi_3(K,\alpha)=\omega_{\partial K}(\alpha),
\end{equation}
for every Borel set $\alpha\subset\partial K$.
\end{theorem}
\begin{remark}\label{rmk_ac}{\rm The measure $\hat H_0^u$ is absolutely continuous with respect to the Lebesgue measure $\lambda$ of $\mathfrak h=\R^2$. Namely, one has for every Borel set $\beta\subset\R^2$, 
\begin{equation}\label{areameasure}
\hat H_0^u(\beta)=\int_\beta \sqrt{e^{4u(\mbx)}+e^{2u(\mbx)}|\nabla_\mbx u|^2}d\lambda(\mbx),
\end{equation}
where $\nabla_\mbx u$ is the gradient of $u$ at $\mbx$ defined for $\lambda$-a.e. point $\mbx$. When $u$ is smooth, formula \eqref{areameasure} can be deduced from \eqref{horometric} via a direct calculation. For general horoconvex function $u$, one can apply \eqref{areameasure} to the approximating sequence of smooth horoconvex functions $u_j$ introduced above, and use the fact that $u_j\to u$ uniformly and $W^{1,2}_{loc}$ in $\B_R$, that $\nabla u$ is in $L^2_{loc}$ since $u$ is Lipschitz, and that the weak distributional gradient of $u$ coincides with the a.e. defined pointwise gradient.

By the way, note that 
\begin{equation}\label{def_areameas}
\hat H^v_0(\beta)=\Phi_{3}(K(v),\mathfrak{g}(v|_\beta))=\mathcal H^2(\mathfrak{g}(v|_\beta))
\end{equation} 
is the 2-dimensional Hausdorff measure of $\mathfrak{g}(v|_\beta)$. This follows from the co-area formula applied to the Lipschitz map $\R^2\to\H^3$ defined by $\mbx\mapsto(\mbx,u(\mbx))_{\mathfrak h}$.

}\end{remark}
\begin{remark}{\rm The generalized Gauss theorem for convex sets in $\H^3$ was already stated in \cite[Chapter XII, p. 397]{Al} and \cite[Section V.2]{Po}. Actually, to define the extrinsic curvature of the boundary of a convex set, Alexandrov used a different approach based on the local Euclidean character of $\H^3$. The validity of \eqref{eq_GClim2} proves also that Alexandrov's definition of extrinsic curvature is equivalent to Kohlmann's definition of the curvature measure $\Phi_1(K,\cdot)$.}\end{remark}

\section{The regular part of the curvature measure}\label{sect_reg}

In this section we will prove that the local curvature at normal points $\mathrm{LocCurv}$ defined in Section \ref{sect_biLip} is in fact the regular part of the curvature measure. Thanks to Gauss equation \eqref{eq_GClim2} here one can consider indifferently (up to a constant factor) the intrinsic curvature measure $\omega_{\partial K}(\cdot)$ or the extrinsic curvature measure $\Phi_1(K,\cdot)$. 
Since both measures $\Phi_1(K,\cdot)$ and $\mathcal H^2$ (restricted to Borel sets) are Radon, according to Lebesgue decomposition theorem one can write 
\[ 
\Phi_1(K,\cdot)=\Phi_1(K,\cdot)_{reg}+\Phi_1(K,\cdot)_{sing}
\]
where the singular part $\Phi_1(K,\cdot)_{sing}$ is supported by a set of  nul $2$-dimensional Hausdorff measure, while $\Phi_1(K,\cdot)_{reg}$ is absolutely continuous with respect to $\mathcal H^{2}$.

\begin{theorem}\label{th_regpart}
Let $K\in \mathcal K(\H^3)$. For every Borel set $\alpha\subset\partial K$ it holds
\begin{equation}\label{eq_regpart}
(\omega_{\partial K})_{reg}(\alpha)=\frac{1}{3}\Phi_1(K,\alpha)_{reg}+\mathcal H^2(\alpha)=\int_\alpha \mathrm{LocCurv}_{\partial K}(q) d\mathcal H^2(q).
\end{equation}
\end{theorem}

Note that $\mathrm{LocCurv}_{\partial K}$ is a.e. defined. Its local integrability will be a byproduct of the proof.

\begin{proof}
The first equality is a direct consequence of \eqref{def_areameas} and Gauss equation \eqref{eq_GClim2}. Clearly we can prove the second equality locally.

As in Section \ref{sect_bdy}, let $u:\B_R\subset \mathfrak{h}\to \R$ be a horoconvex function defined on a horosphere $\mathfrak{h}$ such that its graph $\mathfrak g_\mathfrak{h}(u)$ is an open set of $\partial K$. 
Projecting onto $\mathfrak{h}$, the identity \eqref{eq_regpart} corresponds to
\[
(\hat H_{2}^u)_{reg}(\beta)+\hat H_{0}^u(\beta)= \int_\beta \mathrm{LocCurv}_{\partial K}((\mbx,u(\mbx))_\mathfrak{h}) d\hat H_0^u(\mbx)
\]
for any Borel set $\beta\in \B_R\subset \mathfrak{h}$.
According to a version of the Lebesgue decomposition theorem, \cite[p. 42]{EG}, this latter is equivalent to
\begin{equation}\label{ptwcurv}
\lim_{\epsilon\to 0}\frac{\hat H_{2}^u(\B_\epsilon(\mbx))+\hat H_{0}^u(\B_\epsilon(\mbx))}{\hat H^u_0(\B_\epsilon(\mbx))}=\mathrm{LocCurv}_{\partial K}((\mbx,u(\mbx))_\mathfrak{h})
\end{equation}
for $\hat H^u_0$-a.e. $\mbx\in\B_R$, and in view of Remark \ref{rmk_ac} it is enough to check the latter equality for $\lambda$-a.e. $\mbx\in\B_R$. 
Indeed, we are going to prove that \eqref{ptwcurv} holds at any $\mbx$ at which $u$ is twice differentiable. 

Fix such an $\mbx$. 
Define $\bar u:\mathfrak{h}\to \R$ as in \eqref{def_ubar} so that \eqref{2ndord} holds. Set $q:=(\mbx,u(\mbx))_\mathfrak{h}$ and 
$K(\bar u)$ as in \eqref{def_Kv} so that $\mathfrak{g}_\mathfrak{h}(\bar u|_{\B_R})\subset \partial K(\bar u)$. We have
\begin{equation}\label{asym_haus}
\dist_{\mathcal H}(\mathfrak{g}_\mathfrak{h}(u|_{\B_\epsilon(\mbx)}),\mathfrak{g}_\mathfrak{h}(\bar u|_{\B_\epsilon(\mbx)}))\leq\sup_{\B_\epsilon(\mbx)}|u-\bar u|=o(\epsilon^2).
\end{equation}

%
Fix $\rho>0$ and let $p\in M_\rho(K(u),\mathfrak{g}(u|_{\B_\epsilon(\mbx)}))$. Either $p\in K(\bar u)$ or not. In the first case, since $$\diht(p,K(u))=\diht(p,f_{K(u)}(p))=o(\epsilon^2)$$ and since $f_{K(u)}(p)\in\mathfrak{g}(u|_{\B_\epsilon(\mbx)})$, we get for $\epsilon$ small enough
\begin{equation}\label{1stq}
p\in\mathcal A_{\bar u}^u:=\{(\mathbf y,z)_{\mathfrak{h}}\in \H^3\ :\ \mathbf y\in \B_{2\epsilon}(\mbx)\text{ and }\bar u(\mathbf y)\leq z \leq u(\mathbf y)\}.
\end{equation} 
A direct calculation gives $\vol_{\H^3}(\mathcal A_{\bar u}^u)=o(\epsilon^3).$ On the other hand, if  $p\not \in K(\bar u)$, \eqref{asym_haus} and Lemma \ref{lem_projection} imply
\[\diht(f_{K(u)}(p),f_{K(\bar u)}(p))\leq 2(\tanh \rho)^{1/2}o(\epsilon)+ o(\epsilon^2)=o(\epsilon),
\]
while
\[
\diht(p,f_{K(\bar u)}(p))=\diht(p,K(\bar u))\leq \diht(p,f_{ K(u)}(p))+\diht(f_{ K(u)}(p), K(\bar u))\leq  \rho + o(\epsilon^2).
\]
Combining this latter with \eqref{1stq} we deduce that 
\[
M_\rho(K(u),\mathfrak{g}(u|_{\B_\epsilon(\mbx)}))\subset M_{\rho+o(\epsilon^2)}(K(\bar u),\mathfrak{g}(\bar u|_{\B_{\epsilon+o(\epsilon)}(\mbx)})) \cup \mathcal A_{\bar u}^u,
\]
which in turn implies
\[
\mu_\rho(K(u),\mathfrak{g}(u|_{\B_\epsilon(\mbx)}))\leq \mu_{\rho+o(\epsilon^2)}(K(\bar u),\mathfrak{g}(\bar u|_{\B_{\epsilon+o(\epsilon)}(\mbx)})) + o(\epsilon^3).
\]
Note that for $1\leq k\leq 3$ it holds $\ell_k(\rho+o(\epsilon^2))=\ell_k(\rho)+o(\epsilon^2)$.  Then Theorem \ref{th_kohl} gives 
\[
\sum_{r=0}^{2}\ell_{3-r}(\rho)\Phi_r(K(u),\mathfrak{g}(u|_{\B_\epsilon(\mbx)}))\leq \sum_{r=0}^{2}(\ell_{3-r}(\rho)+o(\epsilon^2))\Phi_r(K(\bar u),\mathfrak{g}(\bar u|_{\B_{\epsilon+o(\epsilon)}(\mbx)})),
\]
that is
\begin{equation}\label{comp_meas1}
\sum_{k=0}^{2}\ell_{k+1}(\rho)\left(\begin{array}{c}
3\\k
\end{array}\right) H_{k}^u(\B_\epsilon(\mbx))\leq \sum_{k=0}^{2}(\ell_{k+1}(\rho)+o(\epsilon^2))\left(\begin{array}{c}
3\\k
\end{array}\right)\hat H_{k}^{\bar u}(\B_{\epsilon+o(\epsilon)}(\mbx)).
\end{equation}
Moreover, interchanging the roles of $u$ and $\bar u$ we also get the converse relation
\begin{equation}\label{comp_meas2}
\sum_{k=0}^{2}\ell_{k+1}(\rho)\left(\begin{array}{c}
3\\k
\end{array}\right)\hat H_{k}^u(\B_\epsilon(\mbx))\geq \sum_{k=0}^{2}(\ell_{k+1}(\rho)-o(\epsilon^2))\left(\begin{array}{c}
3\\k
\end{array}\right)\hat H_{k}^{\bar u}(\B_{\epsilon-o(\epsilon)}(\mbx)).
\end{equation}
Since $\bar u$ is smooth, from \eqref{eq_regphi}, \eqref{def_hatmu} and Lebesgue differentiation theorem we have that the limits
\begin{equation}\label{eq_limits}
h^{\bar u}_k(\mbx):=\lim_{\epsilon\to 0}\frac{\hat H_{k}^{\bar u}(\B_{\epsilon-o(\epsilon)}(\mbx))}{\lambda(\B_{\epsilon}(\mbx))}=\lim_{\epsilon\to 0}\frac{\hat H_{k}^{\bar u}(\B_{\epsilon+o(\epsilon)}(\mbx))}{\lambda(\B_{\epsilon}(\mbx))},\quad 0\leq k\leq 2,
\end{equation}
exist and are finite.
Letting $\epsilon\to 0$, \eqref{comp_meas1} and \eqref{comp_meas2} thus give
\begin{align*}
\sum_{k=0}^{2}\ell_{k+1}(\rho) \left(\begin{array}{c}
3\\k
\end{array}\right)\lim_{\epsilon\to 0}\frac{\hat H_{k}^u(\B_\epsilon(\mbx))}{\lambda(\B_{\epsilon}(\mbx))}
&=
\lim_{\epsilon\to 0}\frac{\sum_{k=0}^{2}\ell_{k+1}(\rho)\left(\begin{array}{c}
3\\k
\end{array}\right)\hat H_{k}^u(\B_\epsilon(\mbx))}{\lambda(\B_{\epsilon}(\mbx))}\\
&=\sum_{k=0}^{2}\ell_{k+1}(\rho)\left(\begin{array}{c}
3\\k
\end{array}\right)h^{\bar u}_k(\mathbf x).
\end{align*}
Since this latter is true for arbitrary small enough positive $\rho$, it yields
\begin{align*}
\lim_{\epsilon\to 0}\frac{\hat H_{k}^u(\B_\epsilon(\mbx))}{\lambda(\B_{\epsilon}(\mbx))}=
h^{\bar u}_k(\mathbf x),\quad 0\leq k\leq 2.
\end{align*}
Hence 
\begin{align*}
\lim_{\epsilon\to 0}\frac{\hat H^u_{2}(\B_\epsilon(\mbx))+\hat H^u_0(\B_\epsilon(\mbx))}{\lambda(\B_{\epsilon}(\mbx))}=
h^{\bar u}_{2}(\mathbf x)+h^{\bar u}_0(\mathbf x),
\end{align*}
which gives
\begin{align}\label{1streg}
\lim_{\epsilon\to 0}\frac{\hat H^u_{2}(\B_\epsilon(\mbx))+\hat H^u_0(\B_\epsilon(\mbx))}{\hat H^u_0(\B_\epsilon(\mbx))}
&=
\lim_{\epsilon\to 0}\frac{\hat H^u_{2}(\B_\epsilon(\mbx))+\hat H^u_0(\B_\epsilon(\mbx))}{\lambda(\B_\epsilon(\mbx))}\lim_{\epsilon\to 0}\frac{\lambda(\B_\epsilon(\mbx))}{\hat H^u_0(\B_\epsilon(\mbx))}\\
&=\frac{h_{2}^{\bar u}(\mbx)}{{h_0^{\bar u}(\mbx)}}+1.\nonumber
\end{align}
On the other hand, by \eqref{eq_GC}, \eqref{eq_limits} and Lebesgue differentiation, 
\begin{align*}
\frac{h_2^{\bar u}(\mbx)}{h_0^{\bar u}(\mbx)}+1&
=\lim_{\epsilon\to 0}\frac{\hat H^{\bar u}_{2}(\B_\epsilon(\mbx))+\hat H^{\bar u}_{0}(\B_\epsilon(\mbx))}{\hat H^{\bar u}_0(\B_\epsilon(\mbx))}\\
&=\lim_{\epsilon\to 0}\frac{\omega_{\partial K(\bar u)}(\mathfrak g_{\mathfrak h}(\bar u|_{\B_\epsilon(\mbx)}))}{\mathcal H^2(\mathfrak g_{\mathfrak h}(\bar u|_{\B_\epsilon(\mbx)}))}\\
&=\lim_{\epsilon\to 0}\frac{\int_{\mathfrak g_{\mathfrak h}(\bar u|_{\B_\epsilon})}\sect_{\partial \bar K}(p) d\mathcal H^2(p)}{\mathcal H^2(\mathfrak g_{\mathfrak h}(\bar u|_{\B_\epsilon(\mbx))})}\\
&=\sect_{\partial \bar K}(q)\\
&=\mathrm{LocCurv}_{\partial K}(q),\end{align*}
which together with \eqref{1streg} prove \eqref{ptwcurv}, and hence \eqref{eq_regpart}.
\end{proof}

In \cite{Machi}, Machigashira studied the (intrinsic) curvature measures of a surface $(S,d)$ which has curvature bounded below by $k\in \R$ in the sense of Alexandrov. Namely, he introduced the so called \textsl{curvature regular points}, which form a set of full $\mathcal H^2$ measure in $S$, and characterized the regular part of the curvature measure in term of the asymptotics of the upper excess at curvature regular points. To this end, he defined the Gaussian curvature function $G$ by
\begin{equation}\label{eq_gauss}
G(x)=\inf_{d>0} \underline{G_d}(x);\qquad\underline{G_d}(x):=\liminf_{\Delta\to \{x\}, x\in\Delta,\Delta\in U_d}\frac{e(\Delta)}{Area(\Delta)}.
\end{equation}
and proved that 
\begin{equation*}
(\omega_{(S,d)})_{reg}(\alpha)=\int_\alpha G(q) d\mathcal H^2(q).
\end{equation*}
The liminf in \eqref{eq_gauss} is taken for geodesic triangles converging to the point $x$, containing $x$ in their interior, and with interior angles greater than $d>0$. Moreover the excess $e(\Delta)$ of the geodesic triangle $\Delta$ is defined so that $e(\Delta)+\pi$ is the sum of the interior angles of $\Delta$.

As discussed above, according to a well-known result by Alexandrov any $CBB(k)$ surface is locally isometric to an open set of the boundary $\partial K$ of a convex set $K$ in the simply connected space form of constant curvature $k$. Clearly, up to a rescaling of the metric, we can always suppose that one has $k\geq -1$ locally, and hence take $K\subset\H^3$ and $\omega_{(S,d)}=\omega_{\partial K}$. Either $K\in \mathcal K(\H^3)$, i.e. it has non-empty interior, or $K$ is isometric to the double of a convex set $\mathfrak K$ of $\H^2$. In the first case, Theorem \ref{th_regpart} applies to $(S,d)$ and a further characterization of the regular part of the curvature measure of a $CBB(k)$ surface is given by the relation
\begin{equation}\label{eq_Machi}
G(q)=\mathrm{LocCurv}_{\partial K}(q).
\end{equation}
In the second case, all points but the ones in $\partial \mathfrak K$, hence up to a set of nul $\mathcal H^2$-measure, have a hyperbolic neighborhood, so that for those points trivially
\[
G(q)=-1=\mathrm{LocCurv}_{\partial K}(q).\]
\begin{remark}
The relation \eqref{eq_Machi}
together with the biLipschitz approximation around normal points given by Corollary \ref{coro_biLip}, will be used in \cite{Ve} to get a further characterization of the regular part of $CBB(K)$ surfaces via the asymptotic expansion of the extent around normal points.
\end{remark}

\textbf{Acknowledgement.} I would like to thank F. Fillastre for many useful discussions, especially concerning Euclidean convex sets, as well as for reading a first version of this article. I'm indebted to J. Bertrand for suggesting the application to $CBB(k)$ surfaces and for pointing me out Alexandrov's theorem on the a.e. second differentiability of convex functions.

This research has been conducted as part of the project Labex MME-DII (ANR11-LBX-0023-01). The author is member of the ``Gruppo Nazionale per l'Analisi Matematica, la Probabilit\'a e le loro Applicazioni'' (GNAMPA) of the \textsl{Istituto Nazionale di Alta Matematica} (INdAM).

\bibliographystyle{alpha}
\bibliography{scalar}
%
%
%
%
%
%

\end{document}